\documentclass[12pt]{amsart}

\usepackage{latexsym}
\usepackage{psfrag}
\usepackage{amsmath}
\usepackage{amssymb}
\usepackage{epsfig}
\usepackage{amsfonts}
\usepackage{amscd}
\usepackage{mathrsfs}
\usepackage{graphicx}
\usepackage{enumerate}

\newtheorem{theorem}{Theorem}
\newtheorem{lemma}[theorem]{Lemma}

\newtheorem{cor}[theorem]{Corollary}

\newtheorem{defn}[theorem]{Definition}

\dedicatory{Dedicated to Bruce C. Berndt for his 80th birthday.}

\begin{document}

\title[Two formulas in Slater's List]
{A combinatorial construction \\ for two formulas in Slater's List}

\author[Kur\c{s}ung\"{o}z]{Ka\u{g}an Kur\c{s}ung\"{o}z}
\address{Faculty of Engineering and Natural Sciences, Sabanc{\i} University, \.{I}stanbul, Turkey}
\email{kursungoz@sabanciuniv.edu}

\subjclass[2010]{05A17, 05A15, 11P84}

\keywords{integer partition, partition generating function, Rogers-Ramanujan identities, Slater's list}

\date{November 2019}

\begin{abstract}
\noindent
We set up a combinatorial framework 
for inclusion-exclusion on the partitions into distinct parts
to obtain an alternative generating function of partitions into distinct and non-consecutive parts.  
In connection with Rogers-Ramanujan identities, 
the generating function yields two formulas in Slater's list.  
The same formulas were constructed by Hirschhorn.  
Similar formulas were obtained by Bringmann, Mahlburg and Nataraj.  
We also use staircases to give alternative triple series 
for partitions into $d-$distinct parts for any $d \geq 2$.  
\end{abstract}

\maketitle


\section{Introduction} 
\label{secIntro}

Number 19 in Slater's list~\cite{Slater-GG} is 
\begin{align}
\label{eqSlater19}
  (-q; q)_\infty \sum_{n \geq 0} \frac{(-1)^j q^{3j^2}}{ (q^2; q^2)_j (-q; q)_{2j} } 
  = \frac{ 1 }{ (q; q^5)_\infty (q^4; q^5)_\infty }.  
\end{align}

Here, we use the $q$-Pochhammer symbols~\cite{GR}
\begin{align*}
 (a; q)_n = \prod_{j = 1}^{n} (1 - aq^{j-1})
\end{align*}
for $j \in \mathbb{N} \cup \{ \infty \}$.  

The companion to \eqref{eqSlater19} in~\cite{Slater-GG} is Number 15: 
\begin{align}
\label{eqSlater15original}
  (-q; q)_\infty \sum_{n \geq 0} \frac{(-1)^j q^{3j^2 - 2j }}{ (q^2; q^2)_j (-q; q)_{2j} } 
  = \frac{ 1 }{ (q^2; q^5)_\infty (q^3; q^5)_\infty }.  
\end{align}

In his PhD thesis Chapter 5~\cite{Hirschhorn-thesis}, 
Hirschhorn gave a combinatorial construction of \eqref{eqSlater19} 
and
\begin{align}
\label{eqSlater15alt}
  (-q; q)_\infty \sum_{n \geq 0} \frac{(-1)^j q^{3j^2 + 2j}}{ (q^2; q^2)_j (-q; q)_{2j+1} } 
  = \frac{ 1 }{ (q; q^5)_\infty (q^4; q^5)_\infty }.  
\end{align}
He showed the equivalence of \eqref{eqSlater15original} and \eqref{eqSlater15alt}, as well.  
He elaborated on the \emph{length of runs} (called sequences in~\cite{Hirschhorn-thesis})
in a partition into distinct parts.  

In this note, 
we define certain moves on partitions into distinct parts, 
and apply inclusion-exclusion on the number of runs of length at least two 
to obtain the formulas \eqref{eqSlater19} and \eqref{eqSlater15alt}.  
By adding staircases, we show that it is possible to get alternative triple series 
as generating functions of partitions into $d$-distinct parts.  

Section \ref{secDef} has the necessary definitions and $q$-series formulas we will use.  
In Section \ref{secMain}, we give the main construction.  
Section \ref{secShiftAndStaircases} discusses how to enhance the construction 
and obtain the desired $q$-series identities.  

It should be also noted that Bringmann, Mahlburg and Nataraj~\cite{BMN} 
found a generating function for partitions into distinct parts 
without $k$-sequences (runs of length $k$) for $k \geq 2$:  
\begin{align}
\label{genFuncBMNptnDistKRunFree}
  \mathcal{C}_k(x; q)
  = \sum_{j, r \geq 0} 
    \frac{ (-1)^j x^{kj + r} q^{ \frac{(r+kj)(r + kj + 1)}{2} + k \frac{j(j-1)}{2} } }
      { (q^k; q^k)_j (q; q)_r }.  
\end{align}
Should we expand this as a double power series in $x$ and $q$, 
the exponent of $q$ is the number being partitioned, and the exponent of $x$ is the number of parts.  
Obviously, $\mathcal{C}_2(1; q)$ is identical to \eqref{eqSlater19}, 
and $\mathcal{C}_2(q; q)$ to \eqref{eqSlater15original}.  
The combinatorial construction of \eqref{genFuncBMNptnDistKRunFree} in~\cite{BMN}
is reminiscent of Hirshhorn's construction of \eqref{eqSlater19}, 
but they have an alternative proof by solving $q-$difference equations, 
asymptotic formulas for the enumerants they consider, 
as well as a fairly extensive list of references.

\section{Definitions and Auxiliary Formulas} 
\label{secDef} 

A partition of a positive integer $n$ 
is a finite non-decreasing sequence of non-negative integers whose sum is $n$.  
\begin{align*}
 \lambda: \lambda_1 + \lambda_2 + \cdots + \lambda_l = n, 
\end{align*}
where $0 \leq \lambda_1 \leq \lambda_2 \leq \cdots \leq \lambda_l$.  
The number being partitioned, $n$, is called the \emph{weight} of the partition, 
denoted by $\vert \lambda \vert$.  
The number of parts, $l$, is called the length of the partition, 
denoted by $l(\lambda)$.  
Allowing zeros in a partition does not change the weight, 
but it changes the length.  

Given a positive integer $d$, 
a partition (into positive parts) is called $d$-distinct if 
$\lambda_{j} - \lambda_{j-1} \geq d$ for $j = 2, 3, \ldots, l$.  
$1$-distinct partitions are just partitions into distinct parts.  
2-distinct partitions are partitions into distinct and non-consecutive parts.  

In a partition into distinct parts, 
a maximal streak of consecutive parts is called a \emph{run}.  
For example, the partition $1, 2, 3, 5, 7, 8, 9$ has three runs: 
$\{ 1, 2, 3 \}, \{ 5 \}, \{ 7, 8, 9 \}$.  

\begin{defn}
\label{defRaft}
  A \emph{raft} in a partition into distinct parts is 
  a pair of consecutive parts $[k, k+1]$ such that $k+2$ is not a part.  
\end{defn}

In other words, a raft is the largest two parts in a run with two or more parts.  
Please observe the maximum number of rafts in a partition into distinct parts 
is the number of runs with two or more parts.  
Not all largest pairs in runs need to be designated as rafts.  
For instance, the partition $1, 2, 3, 5, 7, 8, 9$ could have 
\begin{enumerate}[a)]
 \item no designated rafts as $1, 2, 3, 5, 7, 8, 9$, 
 \item one designated raft as $1, [2, 3], 5, 7, 8, 9$ 
  or $1, 2, 3, 5, 7, [8, 9]$, 
 \item  or two designated rafts as $1, [2, 3], 5, 7, [8, 9]$.  
\end{enumerate}

In the last instance, the rafts $[2, 3]$ and $[8, 9]$ can be compared in the obvious way.  
The former will be called the smaller, and the latter the larger.  
Because there are no other designated rafts between them, 
those two rafts will be called $successive$ rafts.  

For convenience, we indicate the designated rafts by square brackets around them.  

If the next smallest part after the raft $[k, k+1]$ is at least $k+4$, 
that is, if the way is \emph{clear ahead}, 
then the raft can move forward as follows.  

\begin{align*}
  (\textrm{parts } \leq k-1), [k, k+1], (\textrm{parts } \geq k+4)
\end{align*}
\begin{align}
\label{eqSimpleFwdMove}
  \Bigg\downarrow \textrm{ one forward move} 
\end{align}
\begin{align*}
  (\textrm{parts } \leq k-1), [k+1, k+2], (\textrm{parts } \geq k+4)
\end{align*}

Instead, if there is a run of length $s\geq 1$ 
containing $k+3$ with no designated raft at the end, 
then the raft moves forward as follows.  

\begin{align*}
  (\textrm{parts } \leq k-1), [k, k+1], k+3, k+4, \ldots, k+s+2, (\textrm{parts } \geq k+s+4)
\end{align*}
\begin{align}
\label{eqFwdMoveAndAdjustment}
  \Bigg\downarrow \textrm{ one forward move} 
\end{align}
\begin{align*}
  (\textrm{parts } \leq k-1), [k+1, k+2], k+3, k+4, \ldots, k+s+2, (\textrm{parts } \geq k+s+4)
\end{align*}
\begin{align*}
  \Bigg\downarrow \textrm{ rearranging} 
\end{align*}
\begin{align*}
  (\textrm{parts } \leq k-1), k+1, k+2, \ldots [k+s+1, k+s+2], (\textrm{parts } \geq k+s+4)
\end{align*}

With or without an obstacle, 
a forward move increases the weight of the partition by two.  
The rearrangement does not change the weight.  
Backward moves as inverses of forward moves are defined analogously.  

The careful reader will have noticed that Definition \ref{defRaft} 
stipulates that rafts do not \emph{collide} nor are they \emph{docked to the same platform}.  
That is, there can be at most one designated raft in a run.  
For instance, 
\begin{align*}
 \ldots,[k, k+1], [k+2, k+3], \ldots
\end{align*}
is not an admissible configuration.  
We will have more to say about this requirement after the proof of Theorem \ref{thmNoRaftGenFunc}.  

Below are some formulas from~\cite{GR} that we are going to use in the proofs.  
\begin{align}
\label{eqQBinomialLimiting}
  (-z; q)_\infty = & \sum_{n \geq 0} \frac{q^{n(n-1)/2} z^n}{ (q; q)_n } \\
\label{eqQBinomial}
  \frac{ (az; q)_\infty }{ (z; q)_\infty } = & \sum_{n \geq 0} \frac{ (a; q)_n z^n}{ (q; q)_n } \\ 
\label{eqQGauss}
  \sum_{n \geq 0} \frac{ (a; q)_n (b; q)_n (c/ab)^n }{ (q; q)_n (c; q)_n } 
  = & \frac{ (c/a; q)_\infty (c/b; q)_\infty }{ (c; q)_\infty (c/ab; q)_\infty  }
\end{align}

\section{Main Results}
\label{secMain}

The proof of the following lemma is straightforward, 
and left to the reader.  

\begin{lemma}
\label{lemmaMovesAllowMoves}
  Let $[k, k+1]$ and $[l, l+1]$ be successive rafts 
  in a given partition into distinct parts where $k < l$.  
  \begin{enumerate}[(i)]
   \item When conditions exist, a forward move on $[l, l+1]$ allows at least one forward move 
    on $[k, k+1]$.  
   \item When conditions exist, a backward move on $[k, k+1]$ allows at least on backward move 
    on $[l, l+1]$.  
  \end{enumerate}
\end{lemma}

\begin{lemma}
\label{lemmaMinimalPtnGenFunc}
  Let $\beta$ be a minimal partition into distinct parts 
  having exactly $k$ designated rafts for $k \geq 1$.  
  That is, no further backward moves are possible on any of the designated rafts of $\beta$.  
  Then, a generating function of such $\beta$'s is 
  \begin{align}
  \label{eqMinimalPtnGenFunc}
    \sum_{\beta} q^{\vert \beta \vert}
    = \sum_{m \geq 0} q^{\binom{3k+m}{2} - 3\binom{k}{2}} 
    \begin{bmatrix} m+k-1 \\ k-1 \end{bmatrix}_{q^{-1}}
    (-q^{3k+m+1}; q)_{\infty}.  
  \end{align}

\end{lemma}

\begin{proof}
 Let $[r_1, r_1 + 1]$ be the smallest raft in $\beta$.  
 Then, $\beta$ must contain the parts $1, 2, \ldots, r_1-1$.  
 Any missing part among those would have allowed a backward move on $[r_1, r_1 + 1]$.  
 In addition, $r_1 + 1$ cannot be a part in $\beta$.  
 
 Induction on $k$ gives $\beta$ as 
 \begin{align*}
  \beta = & 1, 2, \ldots, r_1 - 1, [r_1, r_1 + 1], \\
  & r_1 + 3, r_1 + 4, \ldots, r_2 - 1, [r_2, r_2 + 1], \\
  & \vdots \\
  & r_{k-1}+3, r_{k-1} + 4, \ldots, r_k - 1, [r_k, r_k+1], \\
  & (\textrm{ parts } \geq r_k+3 \textrm{ containing no rafts}).  
 \end{align*}
 Here, $r_j \geq r_{j-1} + 3$ for $j = 2, 3, \ldots, k$, so $r_k \geq 3k-2$.  
 
 In order to have $k$ rafts among the parts $1, 2, \ldots, r_k+1$, 
 we need $k-1$ missing parts.  
 The missing parts of the displayed $\beta$ above are 
 $r_1 + 2$, $r_2 + 2$, \ldots, $r_{k-1} + 2$.  
 Since $r_{j-1} + 2 \leq r_j - 1$, the missing parts are pairwise at least three apart.  
 
 For a moment, we focus on these missing parts.  
 To generate them, we start with 
 $3$, $6$, \ldots, ${3k-3}$, 
 and add $\mu_1$ to $3k-3$, $\mu_2$ to $3k-6$ etc.  
 $r_{k-1} = 3k-3 + \mu_1 \leq r_k-1$, so $\mu_1 \leq r_k - 3k - 2$.  
 To retain the difference conditions between $\mu$'s, 
 we must have $\mu_1 \geq \mu_2 \geq \cdots \mu_{k-1} \geq 0$.  
 Thus, $\mu$'s form a restricted partition into at most $k-1$ parts, 
 all of which are at most $r_k - 3k + 2$.  
 As such, the $\mu$'s are generated by the Gaussian polynomial 
 $\begin{bmatrix} r_k - 3k + 2 + k - 1 \\ k - 1 \end{bmatrix}$.  
 Thus, the parts of the complementary part of $\beta$ are generated by
 \begin{align}
 \label{eqGenFuncBetaInitMissing}
  q^{3\binom{k}{2}}\begin{bmatrix} r_k - 2k + 1 \\ k - 1 \end{bmatrix}.  
 \end{align}
 
 To obtain the generating function of the portion of $\beta$ 
 containing parts that are at most $r_k+1$, 
 we need to extract the missing parts from among 
 $1$, $2$, \ldots, $r_k+1$.  
 This amounts to replacing $q$ by $q^{-1}$ in \eqref{eqGenFuncBetaInitMissing}, 
 and multiplying by $q^{\binom{r_k + 2}{2}}$: 
 \begin{align}
 \label{eqGenFuncBetaInit}
  q^{\binom{r_k+2}{2} - 3\binom{k}{2}}\begin{bmatrix} r_k - 2k + 1 \\ k - 1 \end{bmatrix}_{q^{-1}}.  
 \end{align}
 
 The parts of $\beta$ that are at least $r_k+3$ merely form a partition into distinct parts.  
 They are generated by $(-q^{r_k + 3}; q)_{\infty}$, 
 to be multiplied by \eqref{eqGenFuncBetaInit}.  
 Then, summing over $r_k \geq 3k-2$ 
 and the change of parameter $r_k = m + 3k - 2$ will conclude the proof.  
\end{proof}

{\bf Remark: } If a partition into distinct parts has no designated rafts, 
it is generated by $(-q; q)_{\infty}$.  
This is not the $k = 0$ case of Lemma \ref{lemmaMinimalPtnGenFunc}, 
so it has to be handled separately.  

\begin{cor}
\label{corPtnGenFunc}
 Let $\lambda$ be a partition into distinct parts 
 having exactly $k$ designated rafts for $k \geq 1$.  
 A generating function for such $\lambda$ is 
 \begin{align}
 \label{eqPtnGenFunc}
    \sum_{\lambda} q^{\vert \lambda \vert}
    = \sum_{m \geq 0} q^{\binom{3k+m}{2} - 3\binom{k}{2}} 
    \begin{bmatrix} m+k-1 \\ k-1 \end{bmatrix}_{q^{-1}}
    \frac{(-q^{3k+m+1}; q)_{\infty}}{ (q^2; q^2)_k }.  
 \end{align}
\end{cor}

\begin{proof}
 The only difference of this generating function from \eqref{eqMinimalPtnGenFunc} 
 is the factor $1/(q^2; q^2)_k$.  
 
 This factor generates a partition into $k$ even parts (allowing zeros)
 $\eta_1 \geq \eta_2 \geq \cdots \geq \eta_{k} \geq 0$.  
 Beginning with a minimal partition $\beta$ into distinct parts with $k$ designated rafts, 
 we move the largest raft forward $\eta_1/2$ times, the next largest $\eta_2/2$ etc. in this order.  
 Thanks to Lemma \ref{lemmaMovesAllowMoves} and the comparison between $\eta$'s, 
 all moves are admissible.  
 This will give us a $\lambda$ as described in the statement.  
 
 Conversely, given such $\lambda$, 
 we perform $\eta_{k}/2$ backward moves on the smallest raft 
 so that no further backward moves on it are possible.  
 This uniquely determines even $\eta_k$.  
 Then, we perform $\eta_{k-1}/2$ backward moves on the next smallest raft 
 so that no further backward moves on it are possible.  
 This uniquely determines even $\eta_{k-1}$.  
 Because of Lemma \ref{lemmaMovesAllowMoves}, $\eta_{k-1} \geq \eta_k$.  
 We continue with the next smallest raft etc. 
 to eventually obtain the partition $\eta$'s into $k$ even parts (allowing zeros).  
 
 The fact that forward and backward moves are inverses of each other, 
 and that we perform the moves in the exact reverse order finishes the proof.  
\end{proof}

\begin{theorem}
\label{thmNoRaftGenFunc}
  \begin{align}
  \label{eqNoRaftGenFunc}
    (-q; q)_{\infty}
    + \sum_{k \geq 1} (-1)^k \sum_{m \geq 0} q^{\binom{3k+m}{2} - 3\binom{k}{2}} 
    \begin{bmatrix} m+k-1 \\ k-1 \end{bmatrix}_{q^{-1}}
    \frac{(-q^{3k+m+1}; q)_{\infty}}{ (q^2; q^2)_k }.  
  \end{align}
  generates partitions into distinct parts in which there can be no designated rafts.  
\end{theorem}

{\bf Remark: } If we cannot designate a raft in a partition into distinct parts, 
then that partition has no runs.  
In particular, it does not have any consecutive parts.  
Thanks to the first Rogers-Ramanujan Identity, we have 
\begin{align}
  \nonumber
  (-q; q)_{\infty}
  + & \sum_{k \geq 1} (-1)^k \sum_{m \geq 0} q^{\binom{3k+m}{2} - 3\binom{k}{2}} 
  \begin{bmatrix} m+k-1 \\ k-1 \end{bmatrix}_{q^{-1}}
  \frac{(-q^{3k+m+1}; q)_{\infty}}{ (q^2; q^2)_k } \\
\label{eqRRalt1}
  & = \frac{1}{(q; q^5)_{\infty} (q^4; q^5)_{\infty}}
\end{align}

\begin{proof}[proof of Theorem \ref{thmNoRaftGenFunc}]
 This is a standard inclusion-exclusion argument.  
 The first term $(-q; q)_{\infty}$ generates partitions into distinct parts, 
 which we will interpret as partitions with no designated rafts.  
 These are weighted by $(-1)^0$.  
 The double sum generates partitions into distinct parts, 
 weigted by $(-1)^k$, where $k \geq 1$ is the number of designated rafts in them.  
 
 The number of runs in a partition into distinct parts, say $K$,
 is the maximum possible number of designated rafts the partition can have.  
 If $K = 0$, then the partition can have no designated rafts.  
 If $K > 0$, then we can designate $k$ rafts in $\binom{K}{k}$ different ways, 
 each of which is weighted by $(-1)^k$.  
 In total, the partition is counted 
 \begin{align*}
  \sum_{k = 0}^K \binom{K}{k} (-1)^k = (1 - 1)^K = 0
 \end{align*}
 times, yielding the proof.  
\end{proof}

If we allowed configurations such as $[k,k+1],[k+2, k+3]$, 
then, the partition $1, 2, 3, 4$ would have been counted 
+1 times as $1, 2, 3, 4$; 
-1 times as $1, 2, [3, 4]$; 
and +1 times as $[1, 2], [3, 4]$; 
a total of +1 times.   
Yet, it should have been annihilated by inclusion-exclusion.  
If we tried to solve this problem, insisting that rafts could be adjacent, 
then there will be other problems either making the moves or raft designation ambiguous, 
or making the generating functions unnecessarily complicated.  

%

\begin{proof}[proof of \eqref{eqSlater19}]
 We'll manipulate the left hand side of \eqref{eqNoRaftGenFunc}.  
 In 
 \begin{align*}
  (-q; q)_{\infty} + \sum_{k \geq 1} (-1)^k 
    \sum_{m \geq 0} q^{ \binom{3k+m}{2} - 3 \binom{k}{2} } 
      \begin{bmatrix} m + k - 1 \\ k - 1 \end{bmatrix}_{q^{-1}}
      \frac{ (-q^{3k+m+1}; q)_{\infty} }{ (q^2; q^2)_k }, 
 \end{align*}
 utilize 
 \begin{align*}
  \begin{bmatrix} m + k - 1 \\ k - 1 \end{bmatrix}_{q^{-1}} 
  = q^{-m(k-1)} \begin{bmatrix} m + k - 1 \\ k - 1 \end{bmatrix} 
  = q^{-m(k-1)} \frac{ (q^k; q)_m }{ (q; q)_m },
 \end{align*}
 as well as 
 \begin{align*}
  (-q^{3k+m+1}; q)_{\infty} = \frac{ (-q; q)_\infty }{ (-q; q)_{3k+m} },
 \end{align*}
 to obtain
 \begin{align*}
  (-q; q)_\infty + (-q; q)_\infty \sum_{k \geq 1} \frac{ (-1)^k q^{3k^2} }{ (q^2; q^2)_k }
    \sum_{m \geq 0} \frac{ q^{ \binom{m}{2} + 2km + m } (q^k; q)_m }{ (q; q)_m (-q^{3k+1}; q)_m }.  
 \end{align*}
 \begin{align*}
  = (-q; q)_\infty + (-q; q)_\infty \sum_{k \geq 1} \frac{ (-1)^k q^{3k^2} }{ (q^2; q^2)_k (-q; q)_{3k} }
    \lim_{b \to 0} \sum_{m \geq 0} \frac{ ( -1/b; q)_m (q^k; q)_m }{ (q; q)_m (-q^{3k+1}; q)_m }
      \left( bq^{2k+1} \right)^m.  
 \end{align*}
 Using the $q$-Gauss summation \eqref{eqQGauss}, this transforms to
 \begin{align*}
  = (-q; q)_\infty + (-q; q)_\infty \sum_{k \geq 1} \frac{ (-1)^k q^{3k^2} }{ (q^2; q^2)_k (-q; q)_{3k} }
    \lim_{b \to 0} \frac{ (bq^{3k+1}; q)_\infty (-q^{2k+1}; q)_\infty }
      { (-q^{3k+1}; q)_\infty (bq^{2k+1}; q)_\infty }  
 \end{align*}
 \begin{align*}
  = (-q; q)_\infty + (-q; q)_\infty \sum_{k \geq 1} \frac{ (-1)^k q^{3k^2} }{ (q^2; q^2)_k (-q; q)_{3k} }
    \frac{ ( -q^{2k+1}; q)_\infty }
      { (-q^{3k+1}; q)_\infty } 
 \end{align*}
 \begin{align*}
  = (-q; q)_\infty + (-q; q)_\infty \sum_{k \geq 1} \frac{ (-1)^k q^{3k^2} }{ (q^2; q^2)_k (-q; q)_{2k} }
 \end{align*}
 \begin{align*}
  = (-q; q)_\infty \sum_{k \geq 0} \frac{ (-1)^k q^{3k^2} }{ (q^2; q^2)_k (-q; q)_{2k} }.  
 \end{align*}
 Theorem \ref{thmNoRaftGenFunc} finishes the proof.  
\end{proof}

\section{Shifts and Staircases}
\label{secShiftAndStaircases}

It is possible to keep track of the number of parts in the processes described in Section \ref{secMain}.  
The proofs of the respective results apply \emph{mutatis mutandis}.  
We record the intermediate formulas here for convenience.  

Let $\beta$ be a minimal partition described in Lemma \ref{lemmaMinimalPtnGenFunc}.  
Then, 
\begin{align*}
 \sum_{\beta} q^{\vert \beta \vert} x^{l(\beta)} 
 = \sum_{m \geq 0} q^{ \binom{3k+m}{2} - 3 \binom{k}{2} } x^{2k+m}
  \begin{bmatrix} m+k-1 \\ k-1 \end{bmatrix}_{q^{-1}} 
  (-xq^{3k+m+1}; q)_\infty.  
\end{align*}

Let $\lambda$ be a partition described in Corollary \ref{corPtnGenFunc}.  
Then, 
\begin{align*}
 \sum_{\lambda} q^{\vert \lambda \vert} x^{l(\lambda)} 
 = \sum_{m \geq 0} q^{ \binom{3k+m}{2} - 3 \binom{k}{2} } x^{2k+m}
  \begin{bmatrix} m+k-1 \\ k-1 \end{bmatrix}_{q^{-1}} 
  \frac{(-xq^{3k+m+1}; q)_\infty}{ (q^2; q^2)_k }.  
\end{align*}

Upon accounting for the number of parts, \eqref{eqRRalt1} becomes
\begin{align*}
 (-xq; q)_\infty + \sum_{k \geq 1} (-1)^k 
  \sum_{m \geq 0} q^{ \binom{3k+m}{2} - 3 \binom{k}{2} } x^{2k+m}
  \begin{bmatrix} m+k-1 \\ k-1 \end{bmatrix}_{q^{-1}} 
  \frac{(-xq^{3k+m+1}; q)_\infty}{ (q^2; q^2)_k } 
  = \sum_{n \geq 0} \frac{ q^{n^2} x^n }{ (q; q)_n }.  
\end{align*}
Then, we can go through the same line of reasoning as in the proof of Theorem \ref{thmNoRaftGenFunc}
to obtain 
\begin{align}
\label{eqSlaterMaster}
  (-xq; q)_\infty \sum_{k \geq 0} \frac{ (-1)^k q^{3k^2} x^{2k} }{ (q^2; q^2)_k (-xq; q)_{2k} } 
  = \sum_{n \geq 0} \frac{ q^{n^2} x^n }{ (q; q)_n }.  
\end{align}
In \eqref{eqSlaterMaster}, if we substitute $x = q$ and 
appeal to the second Rogers-Ramanujan identity~\cite{RR}, 
we arrive at \eqref{eqSlater15alt}.  

%

If we expand the left hand side of \eqref{eqSlaterMaster} as a power series in $x$, 
it will be possible to insert staircases into partitions.  
Using \eqref{eqQBinomialLimiting} and \eqref{eqQBinomial}, 
a generating function of partitions into parts with pairwise difference at least two is
\begin{align*}
 \sum_{m, n, k \geq 0} \frac{ q^{ \binom{n+1}{2} } }{ (q; q)_n }
 \cdot \frac{ (-1)^k q^{3k^2} }{ (q^2; q^2)_{k} } 
 \cdot \frac{ (q^{2k}; q)_m (-q)^m }{ (q; q)_m } 
 \cdot x^{ n + 2k + m }.  
\end{align*}
To make the pairwise difference at least $(2+d)$, 
we insert the staircase 
\begin{align*}
 0 + d + 2d + \ldots + (l(\lambda)-1) d
\end{align*}
to each partition $\lambda$.  
In the generating function as a power series in $x$, 
It amounts to replacing $x^N$ by $x^N q^{d\binom{N}{2}}$.  
This operation yields the generating function 
\begin{align}
\label{eqGenFuncPwsDiffTwoPlusD}
 \sum_{m, n, k \geq 0} \frac{ q^{ \binom{n+1}{2} + d \binom{n}{2} } }{ (q; q)_n }
 \cdot \frac{ (-1)^k q^{3k^2 + d \binom{2k}{2}} }{ (q^2; q^2)_{k} } 
 \cdot \frac{ (q^{2k}; q)_m q^{ d \binom{m}{2} } (-q)^m }{ (q; q)_m } 
 \cdot x^{ n + 2k + m } \, q^{2dnk + dnm + 2dkm}  
\end{align}
for partitions into parts with pairwise difference at least $(2+d)$.  
Unfortunately, and as to be expected, 
unless $d = 0$, series-product identities do not seem to help 
reduce the triple sum in \eqref{eqGenFuncPwsDiffTwoPlusD}
into a single sum, or even a double sum.  


\noindent
{\bf Acknowledgements: }
We thank George E. Andrews and Michael D. Hirshhorn for their guidance and help.  

We also thank the anonymous referees for their helpful comments, 
and for pointing out~\cite{BMN}.  

\bibliographystyle{amsplain}

\end{document}